\documentclass[a4paper,12pt]{article}

\usepackage[truedimen,margin=25mm]{geometry}

\usepackage{amsmath}
\usepackage{amssymb}
\usepackage{amsthm}

\usepackage{xcolor}
\usepackage{enumerate}
\usepackage{ulem}
\usepackage[abbrev]{amsrefs}

\newcommand{\R}{\mathbb{R}}

\newcommand{\N}{\mathbb{N}}
\renewcommand{\H}{\mathbb{H}}

\numberwithin{equation}{section}

\newtheorem{theorem}{Theorem}[section]
\newtheorem{lemma}[theorem]{Lemma}
\newtheorem{proposition}[theorem]{Proposition}
\newtheorem{corollary}[theorem]{Corollary}

\theoremstyle{definition}

\newtheorem{remark}[theorem]{Remark}

\title{On the existence of vector solutions to nonlinear Schr\"odinger equations with weak three-wave interaction}
\author{
Tomoharu Kinoshita\thanks{Corresponding author, Department of Mathematics, School of Science and Engineering,Waseda University, 
3-4-1 Ohkubo, Shinjuku-ku, 169-8555, Tokyo, Japan.
e-mail:luzhiqing@akane.waseda.jp}
~and~
Yohei Sato\thanks{Corresponding author, Department of Mathematics, Saitama University,
Shimo-Okubo 255, Sakura-ku Saitama-shi, 338-8570, JAPAN
e-mail: ysato@rimath.saitama-u.ac.jp}
}
\begin{document}

\date{}

\maketitle

\begin{abstract}
We study a nonlinear Schr\"odinger system with three-wave interaction:
\begin{equation*}
	\left\{\begin{aligned}
	& -  \Delta u_1 = f_1(u_1) + \alpha u_2u_3 \quad \text{ in } \R^N, \\
	& -  \Delta u_2 = f_2(u_2) + \alpha u_3u_1 \quad \text{ in } \R^N, \\
	& -  \Delta u_3 = f_3(u_3) + \alpha u_1u_2 \quad \text{ in } \R^N, \\
	& \quad \vec{u}=(u_1,u_2,u_3)\in (H_{\rm rad}^1(\R^N))^3,
	\end{aligned}\right.
\end{equation*}
where $3\leq N\leq 5$, $\alpha\in \R$ and each nonlinearity $f_i(\xi)$ satisfies the Berestycki-Lions conditions.
Let $S_i$ denote the set of all least energy solutions of the scalar equation $-\Delta u = f_i(u)$ in $H_{\rm rad}^1(\R^N)$.
A solution of the systems is called vector if all its components are nontrivial.  
We establish the existence of two distinct families of vector solutions $\{\vec{u}_\alpha\}$ 
with different asymptotic behaviors as $\alpha \to 0$.
One family satisfies ${\rm dist}(\vec{u}_{\alpha},S_1\times S_2\times S_3) \to 0$,
while another satisfies ${\rm dist}(\vec{u}_{\alpha},S_1\times S_2\times \{0\}) \to 0$.
By contrast, we prove that no family of vector solutions satisfies ${\rm dist}(\vec{u}_{\alpha},S_1\times \{0\}\times \{0\}) \to 0$.
Together, these results give a complete description of the asymptotic structure of vector solutions when the three-wave interaction is weak.

\medskip

\noindent
MSC2010: 
35J57, 
35J50, 
35J91  

\noindent
Keywords: coupled nonlinear Schr\"odinger equations; Schr\"odinger systems; three wave interaction; 
Berestycki-Lions nonlinearity;

\end{abstract}

\section{Introduction}
We study the following nonlinear Schr\"odinger system with three-wave interaction:
\begin{equation}\label{eq:1.1}
	\left\{\begin{aligned}
	& -  \Delta u_1 = f_1(u_1) + \alpha u_2u_3 \quad \text{ in } \R^N, \\
	& -  \Delta u_2 = f_2(u_2) + \alpha u_3u_1 \quad \text{ in } \R^N, \\
	& -  \Delta u_3 = f_3(u_3) + \alpha u_1u_2 \quad \text{ in } \R^N, \\
	& \quad \vec{u}=(u_1,u_2,u_3)\in (H_{\rm rad}^1(\R^N))^3.
	\end{aligned}\right.
\end{equation}
Such systems arise naturally in various physical contexts, including nonlinear optics and plasma physics, 
and provide a natural framework for mathematical study.
They have recently been studied extensively, 
see, for example, \cites{A,CC1,CC2,CCO1,CCO2,CO1,KiO1,KO1,KO2,O1,OS,P1}.
Most of these studies focus on the existence and qualitative properties of ground states, 
as well as their stability, under simple nonlinearities such as $f_i(u) = |u|^{p-1}u$.  
When $\alpha>0$ is small, ground states are scalar solutions,  
while for large $\alpha$, they become vector solutions. 
Here, a solution of the system is called {\bf scalar} if exactly one component is nontrivial 
and {\bf vector} if all components are nontrivial.  
Although these systems have been widely studied, to the best of the authors' knowledge,  
existing results for $\alpha$ close to $0$ mainly concern scalar solutions, 
while the construction of vector solutions has not been thoroughly investigated.

The aim of this paper is to establish the existence of vector solutions of \eqref{eq:1.1} for $\alpha$ close to $0$, 
to investigate their asymptotic behavior as $\alpha \to 0$,  
and to treat the nonlinearities $f_i$ under assumptions as general as possible.
In the weak-interaction regime, any vector solution that exists must stay close to certain limits formed by least-energy solutions of the uncoupled equations.
In particular, we focus on vector solutions that remain close to least-energy solutions of the uncoupled equations. 
We completely classify all such solutions and show that their existence or nonexistence depends on 
which components of the limiting profile remain nontrivial.

\medskip

Throughout this paper, we assume that $3\leq N \leq 5$, $\alpha\in \R$.
We denote by $H_{\rm rad}^1(\R^N)$ the subspace of radially symmetric functions in $H^1(\R^N)$ and set $\H=(H_{\rm rad}^1(\R^N))^3$.
Each nonlinearity $f_i$ is a continuous function and satisfies the following conditions:
\begin{itemize}

\item[(f1)] ${\displaystyle \varlimsup_{|\xi| \to \infty}}\frac{f(\xi)}{|\xi|^{2^*-1}} = 0$ where $2^*=\frac{2N}{N-2}$.

\item[(f2)] $-\infty< {\displaystyle \varliminf_{|\xi|\to 0}}\frac{f(\xi)}{\xi}\leq{\displaystyle \varlimsup_{|\xi|\to 0}}\frac{f(\xi)}{\xi}<0$.

\end{itemize}
In addition, we assume that at least one of the $f_i$ satisfies:
\begin{itemize}

\item[(f3)] There exists $\zeta_0>0$ such that $F(\zeta_0)>0$, where $F(\xi)=\int_0^\xi f(\tau)\, d\tau$.

\end{itemize}
As discussed in \cite{BL1}, the assumptions (f1)--(f3) are almost necessary and sufficient for the existence 
of a positive radially symmetric solution of the scalar field equation
\begin{equation*}
	-\Delta u = f(u) \quad \text{ in } \R^N, \qquad u\in H_{\rm rad}^1(\R^N).
\end{equation*}
The functional associated with \eqref{eq:1.1} is defined by 
\begin{equation*}
	I_\alpha(\vec{u}) = \sum_{i=1}^3J_i(u_i) - \alpha \int_{\R^N}u_1u_2u_3\, dx, 
\end{equation*}
where
\begin{equation*}
	J_i(u_i) = \frac{\,1\,}{2}\int_{\R^N} |\nabla u_i|^2\, dx - \int_{\R^N}F_i(u_i)\,dx.
\end{equation*}
When $f_i$ satisfies (f1)--(f3),  the functional $J_i$ has a positive radially symmetric least energy critical point (see \cite{BL1}).
Let $S_i$ denote the set of all radially symmetric least energy critical point of $J_i$.
Then $S_i$ is compact in $H_{\rm rad}^1(\R^N)$ (see Lemma \ref{lem:2.4}).
We use the following notation for norms:
\begin{equation*}
	\|u\|_{L^p}^p = \int_{\R^N} |u|^p \, dx, \qquad
	\|u\|_{H^1}^2 = \|\nabla u\|_{L^2}^2 + \|u\|_{L^2}^2, \qquad
	\|\vec{u}\|_{\H}^2 = \sum_{i=1}^3 \|u_i\|_{H^1}^2.
\end{equation*}
For a subset $A\subset \H$ and a point $\vec{u}\in \H$, we define the distance by
\begin{equation*}
	{\rm dist}(\vec{u},A) = \inf_{\vec{v}\in A}\|\vec{u}-\vec{v}\|_{\H}.
\end{equation*}

We now present the main results of this paper.
We establish the existence of two distinct families of vector solutions $\{\vec{u}_\alpha\}$ 
with different asymptotic behaviors as $\alpha \to 0$.
The first one is given as follows.

\begin{theorem}\label{thm:1.1}
Suppose that $f_1$, $f_2$, $f_3$ satisfy (f1)--(f3).
Then there exists $\alpha_0>0$ such that, for all $\alpha$ with $|\alpha|\leq \alpha_0$, 
\eqref{eq:1.1} has a vector solution $\vec{u}_\alpha$ satisfying
\begin{equation}\label{eq:1.2}
	\lim_{\alpha \to 0}{\rm dist}(\vec{u}_{\alpha},X) = 0 \qquad (X=S_1\times S_2\times S_3).
\end{equation}
\end{theorem}

Since $X$ is compact, \eqref{eq:1.2} implies that there exist a sequence $\alpha_n \to 0$ and $\vec{u}_0\in X$ such that 
$\vec{u}_{\alpha_n}\to \vec{u}_0$ strongly in $\H$.
The second family of vector solutions $\{\vec{u}_\alpha\}$ is described as follows.

\begin{theorem}\label{thm:1.2}
Suppose that $f_1$, $f_2$ satisfy (f1)--(f3) and that $f_3$ satisfies (f1) and (f2).
Then there exists $\tilde{\alpha}_0>0$ such that, for all $\alpha$ with $|\alpha|\leq \tilde{\alpha}_0$, 
\eqref{eq:1.1} has a vector solution $\vec{u}_\alpha$ satisfying
\begin{equation}\label{eq:1.3}
	\lim_{\alpha \to 0}{\rm dist}(\vec{u}_{\alpha},Y) = 0 \qquad (Y=S_1\times S_2\times \{0\}).
\end{equation}
\end{theorem}

Since the solution $\vec{u}_\alpha=(u_{1\alpha},u_{2\alpha},u_{3\alpha})$ provided by Theorem \ref{thm:1.2} stays close to $S_1\times S_2 \times \{0\}$, 
$u_{1\alpha}$ and $u_{2\alpha}$ are necessarily nontrivial.
The three-wave interaction ensures that $u_{3\alpha}$ is nontrivial as well, despite vanishing as $\alpha\to 0$ (see Lemma \ref{lem:4.6}).
Theorem \ref{thm:1.3} below shows that, in contrast, vector solutions near $S_1\times \{0\}\times \{0\}$ do not exist. 
This nonexistence is also a distinctive feature arising from the three-wave interaction.

\begin{theorem}\label{thm:1.3}
Suppose that $f_1$ satisfies (f1)--(f3) and that $f_2$, $f_3$ satisfy (f1) and (f2).
Then \eqref{eq:1.1} has no vector solutions $\vec{u}_\alpha$ satisfying
\begin{equation}\label{eq:1.4}
	\lim_{\alpha \to 0}{\rm dist}(\vec{u}_{\alpha},Z) = 0 \qquad (Z=S_1\times \{0\}\times \{0\}).
\end{equation}
\end{theorem}

\begin{remark}
For any $\alpha\in \R$, the function $(\omega_1,0,0)$ with $\omega_1\in S_1$ is a solution of \eqref{eq:1.1}.
Hence, a scalar solution satisfying \eqref{eq:1.4} always exists.
The arguments developed in this paper for constructing solutions near $X$ and $Y$ also apply to a neighborhood of $Z$,
which allow us to construct solutions near $Z$ as well.
However, Theorem~\ref{thm:1.3} implies that any solution obtained near $Z$ by this approach must necessarily be a scalar solution.
\end{remark}

From Theorems \ref{thm:1.1} and \ref{thm:1.2}, we obtain the multiplicity of vector solutions of \eqref{eq:1.1}.

\begin{corollary}
Suppose that $f_1$, $f_2$, $f_3$ satisfy (f1)--(f3).
Then there exists $\alpha_\ast$ such that, for all $\alpha$ with $|\alpha|\leq \alpha_\ast$, 
\eqref{eq:1.1} has at least four vector solutions.
\end{corollary}

The existence of vector solutions in Theorem \ref{thm:1.2} shows that the three-wave interaction and 
the Bose-Einstein-type interaction lead to different solution structures in the regime where these interactions are small.
The system of coupled nonlinear Schr\"odinger equations with Bose-Einstein-type interactions is given by
\begin{equation}\label{eq:1.5}
	\left\{\begin{aligned}
	& -\Delta u_1 = f_1(u_1) + \beta_{12} u_1 u_2^2 + \beta_{13} u_1 u_3^2 \quad \text{in } \R^N, \\
	& -\Delta u_2 = f_2(u_2) + \beta_{12} u_1^2 u_2 + \beta_{23} u_2 u_3^2 \quad \text{in } \R^N, \\
	& -\Delta u_3 = f_3(u_3) + \beta_{13} u_1^2 u_3 + \beta_{23} u_2^2 u_3 \quad \text{in } \R^N, \\
	& \hspace{1.0cm} \vec{u} = (u_1,u_2,u_3) \in \H.
	\end{aligned} \right.
\end{equation}
In particular, for $\vec{\beta} = (\beta_{12},\beta_{13},\beta_{23})$, 
there exists no family of vector solutions $\{\vec{u}_{\vec{\beta}}\}$ that satisfies
\begin{equation}\label{eq:1.6}
\lim_{|\vec{\beta}| \to 0} \mathrm{dist}(\vec{u}_{\vec{\beta}}, Y) = 0 
\quad \text{or} \quad
\lim_{|\vec{\beta}| \to 0} \mathrm{dist}(\vec{u}_{\vec{\beta}}, Z) = 0 \quad \text{(see Remark \ref{rmk:4.8})}.
\end{equation}

\medskip

The proofs of Theorems \ref{thm:1.1} and \ref{thm:1.2} are based on the method of Byeon--Jeanjean \cite{BJ1}.  
This method was also applied in \cite{CZ2} to construct positive solutions of the linearly coupled system
\begin{equation}\label{eq:1.7}
	\left\{\begin{aligned}
	& -\Delta u_1 = f_1(u_1) + \alpha u_2 \quad \text{in } \R^N, \\
	& -\Delta u_2 = f_2(u_2) + \alpha u_1 \quad \text{in } \R^N, \\
	& \ \ \vec{u} = (u_1,u_2) \in (H_{\rm rad}^1(\R^N))^2.
	\end{aligned}\right.
\end{equation}
In that case, the solutions converge to elements of $S_1 \times S_2$ as $\alpha \to 0^+$.  
Similarly, in \cite{CZ1}, positive solutions converging to elements of $S_1 \times \{0\}$ or $\{0\} \times S_2$ as $\alpha \to 0^+$ were constructed.
Theorems~\ref{thm:1.1}--\ref{thm:1.3} make explicit both the similarities and the differences in the solution structures 
generated by the three-wave interaction $\int_{\R^N} u_1 u_2 u_3 \, dx$ and by the two-wave interaction $\int_{\R^N} u_1 u_2 \, dx$.
We note that the corresponding results in \cites{CZ1,CZ2} required the nonlinearities to satisfy a slightly stronger condition, 
namely $\lim_{|\xi| \to 0} \frac{f(\xi)}{\xi} \in (-\infty,0)$.  

\medskip

This paper is organized as follows.
In Section 2, we prove several preliminary results.
In Section 3, we give the proof of Theorem \ref{thm:1.1}.
In Section 4, we present the proofs of Theorems \ref{thm:1.2} and \ref{thm:1.3}.

\section{Preliminary}
Since we do not assume that $\lim_{|\xi| \to 0} \frac{f_i(\xi)}{\xi} \in (-\infty,0)$ which was used in \cite{CZ2}, 
the standard decomposition of $f$ into its linear and nonlinear parts is not available.
To overcome this difficulty, we adapt an idea from \cite{HIT}.
For each $i=1,2,3$, define
\begin{align*}
	\lambda_i = -\frac{\,1\,}{2}\varlimsup_{|\xi| \to 0} \frac{f_i(\xi)}{\xi}>0, \qquad 
	g_i(\xi) = f_i(\xi) + \lambda_i\xi.
\end{align*}
Then, the system \eqref{eq:1.1} can be rewritten as follows:
\begin{equation}\label{eq:2.1}
	\left\{\begin{aligned}
	& -  \Delta u_1 + \lambda_1 u_1 = g_1(u_1) + \alpha u_2u_3 \quad \text{ in } \R^N, \\
	& -  \Delta u_2 + \lambda_2 u_2 = g_2(u_2) + \alpha u_3u_1 \quad \text{ in } \R^N, \\
	& -  \Delta u_3 + \lambda_3 u_3 = g_3(u_3) + \alpha u_1u_2 \quad \text{ in } \R^N, \\
	& \hspace{1.0cm} \vec{u}=(u_1,u_2,u_3)\in \H.
	\end{aligned}\right.
\end{equation}
Next, define
\begin{equation*}
	h_i^+(\xi) = \begin{cases}
	\max\{g_i(\xi), 0 \} & (\xi \geq 0), \\
	\min\{g_i(\xi), 0 \} & (\xi < 0 ),
	\end{cases} \qquad 
	h_i^-(\xi) = \begin{cases}
	\max\{-g_i(\xi), 0 \} & (\xi \geq 0), \\
	\min\{-g_i(\xi), 0 \} & (\xi < 0 ).
	\end{cases}
\end{equation*}
Then we have
\begin{align*}
	& g_i(\xi) = h_i^+(\xi) - h_i^-(\xi) \quad \text{ for all } \xi \in \R, \\
	& h_i^+(\xi)\xi \geq 0, \quad h_i^-(\xi)\xi \geq 0 \quad \text{ for all } \xi \in \R.
\end{align*}
Moreover, there exists $\nu>0$ such that $h_i^+$ satisfies 
\begin{equation}\label{eq:2.2}
	h_i^+(\xi) = 0  \quad \text{ for all } |\xi|\leq \nu.
\end{equation}
We then have the following lemma.

\begin{lemma}\label{lem:2.1}
Let $\{u_n\}\subset H_{\rm rad}^1(\R^N)$ be a sequence such that $u_n \rightharpoonup u_0$ weakly in $H^1_{\rm rad}(\R^N)$. 
Then the following hold:
\begin{align}
	\lim_{n\to \infty}\int_{\R^N}h_i^+(u_n)u_n\, dx = \int_{\R^N}h_i^+(u_0)u_0\, dx, \label{eq:2.3}\\ 
	\varliminf_{n\to \infty}\int_{\R^N}h_i^-(u_n)u_n\, dx \geq \int_{\R^N}h_i^-(u_0)u_0\, dx. \label{eq:2.4}
\end{align}
\end{lemma}

\begin{proof}
The inequality \eqref{eq:2.4} follows from Fatou's lemma.
Thus, it remains to show \eqref{eq:2.3}.
We use the following fact for $H_{\rm rad}^1(\R^N)$ (see \cite{AM}*{Lemma 11.1} or \cite{BL1}):
there exists a constant $C_N>0$ such that, for all $u\in H_{\rm rad}^1(\R^N)$,
\begin{equation*}
	|u(x)|\leq C_N|x|^{-\frac{N-1}{2}}\|u\|_{H^1(\R^N)} \quad \text{ for all } |x|\geq 1.
\end{equation*}
Since $\{u_n\}$ is bounded in $H_{\rm rad}^1(\R^N)$, we can choose a large $R>0$ such that 
\begin{equation*}
	|u_n(x)|\leq \nu \qquad |x|\geq R, \quad n\in \N\cup \{0\}.
\end{equation*}
Then, from \eqref{eq:2.2}, we have
\begin{equation}\label{eq:2.5}
	h_i^+(u_n(x))=0 \qquad |x|\geq R, \quad n\in \N\cup \{0\}.
\end{equation}
Set $B_R=\{x\in \R^N\,|\, |x|\leq R\}$.
By the compactness of the Sobolev embedding on bounded domains, for every $p\in[1,2^*)$, we have
\begin{equation}\label{eq:2.6}
	u_n \to u_0 \quad  \text{ strongly in } L^p(B_R).
\end{equation}
Take an arbitrary $\epsilon>0$.
From (f1) and \eqref{eq:2.2}, there exists a constant $C_\epsilon>0$ such that
\begin{equation*}
	|h_i^+(\xi)|\leq C_\epsilon |\xi|+\epsilon|\xi|^{2^*-1} \qquad (\xi\in \R).
\end{equation*}
From \eqref{eq:2.5}, we have $h_i^+(u_n)u_n=0$ on $\R^N\setminus B_R$ for all $n\in \N\cup \{0\}$.
Hence
\begin{align*}
	\int_{\R^N}h_i^+(u_n)u_n - h_i^+(u_0)u_0\, dx
	& =\int_{B_R}h_i^+(u_n)u_n - h_i^+(u_0)u_0\, dx \\ 
	& =\int_{B_R}h_i^+(u_n)(u_n-u_0)\, dx + \int_{B_R}(h_i^+(u_n)-h_i^+(u_0))u_0\, dx \\ 
	& =:({\rm I})+({\rm II}).
\end{align*}
From \eqref{eq:2.6}, it follows $\displaystyle \lim_{n\to \infty}({\rm II})=0$.
Since $\|u_n\|_{L^{2^*}(\R^N)}$ is bounded, for a constant $C>0$, we have
\begin{align*}
	\left|({\rm I})\right|
	& \leq \int_{B_R}C_\epsilon |u_n||u_n-u_0|
	+\epsilon |u_n|^{2^*-1}|u_n-u_0|\, dx \\
	& \leq C_\epsilon \|u_n\|_{L^2(B_R)} \|u_n-u_0\|_{L^2(B_R)}
	+\epsilon \|u_n\|_{L^{2^*}(B_R)}^{2^*-1}
	\|u_n-u_0\|_{L^{2^*}(B_R)} \\
	& \leq C_\epsilon \|u_n\|_{L^2(B_R)} \|u_n-u_0\|_{L^2(B_R)}
	+\epsilon C.
\end{align*}
Therefore, from \eqref{eq:2.6}, 
\begin{equation*}
	\varlimsup_{n\to \infty}\left|({\rm I})\right|\leq \epsilon C.
\end{equation*}
Since $\epsilon>0$ is arbitrary, conclude that $\displaystyle \lim_{n\to \infty}({\rm I})=0$.
Hence the proof is complete.
\end{proof}

\begin{proposition}\label{prop:2.2}
Let $\{\vec{u}_n\}\subset \H$ be a bounded sequence satisfying $\|I_{\alpha}'(\vec{u}_n)\|_{\H^*} \to 0$ as $n\to \infty$.
Then there exists a subsequence of $\{\vec{u}_n\}$ that converges strongly in $\H$.
\end{proposition}

\begin{proof}
Let $\vec{u}_n=(u_{1n},u_{2n},u_{3n}) \in \H$ be as in Proposition \ref{prop:2.2}.
Then, up to a subsequence, there exists $\vec{u}_0=(u_{10},u_{20},u_{30})\in \H$ such that
\begin{equation}\label{eq:2.7}
	\vec{u}_n \to \vec{u}_0 \quad \text{ weakly in } \H \text{ and strongly in } (L^p(\R^N))^3 \ (2< p<2^*).
\end{equation}
For any $\vec{\varphi}\in (C_0^\infty(\R^N))^3$, we have $I_{\alpha}'(\vec{u}_n)\vec{\varphi} \to I_\alpha'(\vec{u}_0)\vec{\varphi} = 0$.
Hence $\vec{u}_0$ is a critical point of $I_\alpha$.
Since $I_\alpha'(\vec{u}_n)\vec{u}_n =o_n(1)$ as $n\to \infty$, we have
\begin{align*}
	\sum_{i=1}^3\left(\|\nabla u_{in}\|_{L^2}^2 + \lambda_i\|u_{in}\|_{L^2}^2\right)
	& = \sum_{i=1}^3\int_{\R^N}h_i^+(u_{in})u_{in} - h_i^-(u_{in})u_{in}\,dx \\
	& \hspace{2cm} +3\alpha\int_{\R^N}u_{1n}u_{2n}u_{3n}\,dx + o_n(1).
\end{align*}
From Lemma \ref{lem:2.1} and $I_\alpha'(\vec{u}_0)\vec{u}_0=0$, we obtain
\begin{align*}
	\varlimsup_{n\to \infty} & \sum_{i=1}^3\left(\|\nabla u_{in}\|_{L^2}^2 + \lambda_i\|u_{in}\|_{L^2}^2\right) \\
	& \leq \sum_{i=1}^3\int_{\R^N}h_i^+(u_{i0})u_{i0} - h_i^-(u_{i0})u_{i0}\,dx + 3\alpha\int_{\R^N}u_{10}u_{20}u_{30}\, dx \\
	& = \sum_{i=1}^3\left(\|\nabla u_{i0}\|_{L^2}^2 + \lambda_i\|u_{i0}\|_{L^2}^2\right).
\end{align*}
Therefore, by \eqref{eq:2.7}, we conclude that $\vec{u}_n \to \vec{u}_0$ strongly in $\H$.
\end{proof}

As in Proposition \ref{prop:2.2}, the next result also holds.

\begin{proposition}\label{prop:2.3}
Let $\{\alpha_n\}$ be a sequence with $\alpha_n\to 0$ and 
let $\{\vec{u}_n\}\subset \H$ be a bounded sequence satisfying $\|I_{\alpha_n}'(\vec{u}_n)\|_{\H^*} \to 0$ as $n\to \infty$.
Then there exists a subsequence of $\{\vec{u}_n\}$ that converges strongly in $\H$.
\end{proposition}

\begin{proof}
Let $\vec{u}_n=(u_{1n},u_{2n},u_{3n}) \in \H$ be as in Proposition~\ref{prop:2.3}.
Then, up to a subsequence, there exists $\vec{u}_0=(u_{10},u_{20},u_{30})\in \H$ such that
\begin{equation}\label{eq:2.8}
	\vec{u}_n \to \vec{u}_0 \quad \text{ weakly in } \H \text{ and strongly in } (L^p(\R^N))^3 \ (2< p<2^*).
\end{equation}
For any $\vec{\varphi}\in (C_0^\infty(\R^N))^3$, we have $I_{\alpha_n}'(\vec{u}_n)\vec{\varphi} \to I_0'(\vec{u}_0)\vec{\varphi} = 0$.
Hence $\vec{u}_0$ is a critical point of $I_0$.
Since $I_{\alpha_n}'(\vec{u}_n)\vec{u}_n =o_n(1)$ as $n\to \infty$, we have
\begin{align*}
	\sum_{i=1}^3\left(\|\nabla u_{in}\|_{L^2}^2 + \lambda_i\|u_{in}\|_{L^2}^2\right)
	& = \sum_{i=1}^3\int_{\R^N}h_i^+(u_{in})u_{in} - h_i^-(u_{in})u_{in}\,dx \\
	& \hspace{2cm} +3\alpha_n\int_{\R^N}u_{1n}u_{2n}u_{3n}\,dx + o_n(1).
\end{align*}
From Lemma \ref{lem:2.1} and $I_0'(\vec{u}_0)\vec{u}_0=0$, we obtain
\begin{align*}
	\varlimsup_{n\to \infty} \sum_{i=1}^3\left(\|\nabla u_{in}\|_{L^2}^2 + \lambda_i\|u_{in}\|_{L^2}^2\right) 
	& \leq \sum_{i=1}^3\int_{\R^N}h_i^+(u_{i0})u_{i0} - h_i^-(u_{i0})u_{i0}\,dx \\
	& = \sum_{i=1}^3\left(\|\nabla u_{i0}\|_{L^2}^2 + \lambda_i\|u_{i0}\|_{L^2}^2\right).
\end{align*}
Therefore, by \eqref{eq:2.8}, we conclude that $\vec{u}_n \to \vec{u}_0$ strongly in $\H$.
\end{proof}

We provide a proof of the following lemma to make the paper self-contained.

\begin{lemma}\label{lem:2.4}
$S_i$ is compact in $H_{\rm rad}^1(\R^N)$.
\end{lemma}

\begin{proof}
We first show that $S_i$ is bounded in $H_{\rm rad}^1(\R^N)$.
Let $\omega_i \in S_i$.
Then $\omega_i$ satisfies the Pohozaev identity $\frac{N-2}{2}\|\nabla \omega_i\|_{L^2}^2 = N \int_{\R^N}F_i(\omega_i)\, dx$.
Hence,
\begin{equation*}
	c_i = J_i(\omega_i) 
	= \frac{\,1\,}{2}\|\nabla \omega_i\|_{L^2}^2 - \int_{\R^N}F_i(\omega_i)\, dx 
	= \frac{\,1\,}{N}\|\nabla \omega_i\|_{L^2}^2.
\end{equation*}
Thus the set $\{\|\nabla \omega_i\|_{L^2} \mid \omega_i\in S_i\}$ is bounded.
Since $h_i^+(\xi)$ satisfies \eqref{eq:2.2} and (f1), there exists a constant $C>0$ such that
\begin{equation*}
	|h_i^+(\xi)\xi|\leq C|\xi|^{2^*} \quad \text{ for all } \xi\in \R.
\end{equation*}
Using $J_i'(\omega_i)\omega_i=0$, we obtain
\begin{equation*}
	\|\nabla \omega_i\|_{L^2}^2 + \lambda_i\|\omega_i\|_{L^2}^2 
	\leq \int_{\R^N}h_i^+(\omega_i)\omega_i\,dx 
	\leq C\|\omega_i\|_{L^{2^*}}^{2^*} 
	\leq C'\|\nabla \omega_i\|_{L^2}^{2^*},
\end{equation*}
for a constant $C'>0$.
This shows that $\left\{\|\omega_i\|_{L^2}\,|\,\omega_i\in S_i \right\}$ is bounded.
Consequently, $S_i$ is bounded in $H_{\rm rad}^1(\R^N)$.
Hence, by Proposition~\ref{prop:2.2}, any sequence in $S_i$ has a strongly convergent subsequence.
Therefore, $S_i$ is compact in $H_{\rm rad}^1(\R^N)$.
\end{proof}

\begin{remark}
The compactness of all least energy solutions with a maximum at the origin, 
not necessarily radially symmetric, is also shown in \cite{BJ1}*{Proposition 1}.
\end{remark}

\section{Proof of Theorem \ref{thm:1.1}}

In this section, we prove Theorem \ref{thm:1.1}.
Our approach builds on the methods of \cite{BJ1}, which constructs peak solutions for singularly perturbed problems, 
and \cite{CZ2}, which develops solutions for linearly coupled systems. 
We adapt these techniques to the three-wave interaction setting to establish the existence of vector solutions.

\subsection{Neighborhood of $X=S_1\times S_2\times S_3$}
For $c\in \R$, we define the level set of $I_\alpha$ by 
\begin{equation*}
	[I_\alpha \leq c] = \left\{\vec{u}\in \H \,\left|\, I_\alpha(\vec{u}) \leq c \right.\right\}.
\end{equation*}
Choose a constant $\mu>0$ such that
\begin{equation*}
	0<\mu<\frac{\,1\,}{3}\inf\left\{\left.\|\omega_i\|_{H^1}\,\right|\, \omega_i\in S_i \ (i=1,2,3) \right\}
\end{equation*}
and define the $\mu$-neighborhood of $X$ by
\begin{equation*}
	X_\mu=\{\vec{u} \in \H\,|\, {\rm dist}(\vec{u}, X)\leq \mu\}.
\end{equation*}
By this choice of $\mu$, all components of any element in $X_{2\mu}$ are nontrivial. 
Within $X_{2\mu}$, the following compactness result holds.

\begin{lemma}\label{lem:3.1}
Let $\{\alpha_n\}$ and $\{\delta_n\}$ be sequences with $\alpha_n \to 0$ and $\delta_n \to 0^+$.
Suppose that a sequence $\{\vec{u}_n\}$ satisfies
\begin{equation}\label{eq:3.1}
	\vec{u}_n \in X_{2\mu}\cap [I_{\alpha_n} \leq c_1+c_2+c_3+\delta_n] \quad \text{ and } \quad 
	\|I_{\alpha_n}'(\vec{u}_n)\|_{\H^*}\to 0 \text{ as } n\to \infty.
\end{equation}
Then 
\begin{equation}\label{eq:3.2}
	\lim_{n\to \infty}{\rm dist}(\vec{u}_n, X)=0.
\end{equation}
\end{lemma}

\begin{proof}
Suppose that \eqref{eq:3.2} does not hold.
Then, up to a subsequence, we have
\begin{equation}\label{eq:3.3}
	\lim_{n\to \infty}{\rm dist}(\vec{u}_n, X)>0.
\end{equation}
Since ${X_{2\mu}}$ is bounded, by Proposition \ref{prop:2.3}, up to a further subsequence, 
$\vec{u}_n \to \vec{u}_0$ strongly in $\H$ for some $\vec{u}_0=(u_{10},u_{20},u_{30})$. 
Then $\vec{u}_0\in X_{2\mu}$ and $I_0'(\vec{u}_0)=0$.
Hence each $u_{i0}$ is a nontrivial critical point of $J_i$, and $J_i(u_{i0})\geq c_i$.
Therefore
\begin{equation*}
	\lim_{n\to \infty}I_{\alpha_n}(\vec{u}_n) = \sum_{i=1}^3J_i(u_{i0})\geq c_1+c_2+c_3.
\end{equation*}
Since ${\displaystyle \varlimsup_{n\to \infty}}I_{\alpha_n}(\vec{u}_n)\leq c_1+c_2+c_3$, we deduce that $J_i(u_{i0})=c_i$.
In particular, $\vec{u}_0\in X$, which contradicts \eqref{eq:3.3}.
Thus, \eqref{eq:3.2} holds.
\end{proof}

\begin{proposition}\label{prop:3.2}
There exist $\alpha_1>0$, $\delta_1>0$ and $\rho>0$ such that, for all $\alpha$ with $|\alpha| \leq \alpha_1$, we have
\begin{equation*}
	\|I_\alpha'(\vec{u})\|_{\H^*}\geq \rho \quad \text{ for all } \vec{u}\in (X_{2\mu}\setminus X_{\mu})\cap[I_\alpha\leq c_1+c_2+c_3+\delta_1]
\end{equation*}
\end{proposition}

\begin{proof}
We argue by contradiction. 
Suppose that Proposition \ref{prop:3.2} is false.
Then there exist sequences $\{\alpha_n\}$, $\{\delta_n\}$ with $\alpha_n\to 0$ and $\delta_n\to 0^+$, and a sequence $\{\vec{u}_n\}$ such that
\begin{equation*}
	\vec{u}_n \in X_{2\mu}\cap [I_{\alpha_n}\leq c_1+c_2+c_3+\delta_n] \quad \text{ and } \quad 
	\|I_{\alpha_n}'(\vec{u}_n)\|_{\H^*}\to 0. 
\end{equation*}
From Lemma \ref{lem:3.1}, it follows that $\displaystyle \lim_{n\to \infty}{\rm dist}(\vec{u}_n, X)=0$.
This contradicts the fact that $\vec{u}_n\notin X_\mu$.
Hence Proposition \ref{prop:3.2} holds.
\end{proof}

\subsection{Minimax value}

For $\omega_i\in S_i$, we define the path $\sigma_i:\R\to H^1(\R^N)$ by 
\begin{equation*}
	[\sigma_i(s)](x) = \omega_i(e^{-s}x)  \qquad (s\in \R).
\end{equation*}
We also define the functional $P_i\in C(H^1(\R^N),\R)$ by
\begin{equation*}
	P_i(u) = \frac{N-2}{2}\|\nabla u\|_{L^2}^2 - N\int_{\R^N}F_i(u)\, dx.
\end{equation*}
Since the equality $P_i(u)=0$ is the Pohozaev identity, we have $P_i(\omega_i)=0$.
Moreover, the following holds.

\begin{lemma}\label{lem:3.3}
\begin{enumerate}[(i)]

\item $\sigma_i(0)=\omega_i$ and $\sigma_i(s)>0$ for all $s\in \R$

\item $\displaystyle \lim_{s\to -\infty}\|\sigma_i(s)\|_{H^1} =0$, 
$\displaystyle \lim_{s\to \infty}\|\sigma_i(s)\|_{H^1} =\infty$.

\item $J_i(\sigma_i(s))<c_i=J_i(\sigma_i(0))$ for all $s\not=0$.

\item $P_i(\sigma_i(s))= \frac{N-2}{2}(1-e^{2s})e^{(N-2)s}\|\nabla \omega_i\|_{L^2}^2$. 
In particular, $P_i(\sigma_i(0))=0$.

\end{enumerate}
\end{lemma}

\begin{proof}
(i) is obvious.
By scaling, we have
\begin{equation*}
	\|\sigma_i(s)\|_{H^1}^2 
	= e^{(N-2)s}\|\nabla \omega_i\|_{L^2}^2 + e^{Ns}\|\omega_i\|_{L^2}^2.
\end{equation*}
Hence (ii) follows.
Similarly, by scaling, we obtain
\begin{align*}
	J_i(\sigma_i(s)) 
	& = \frac{\,1\,}{2}e^{(N-2)s}\|\nabla \omega_i\|_{L^2}^2 - e^{Ns}\int_{\R^N}F_i(\omega_i)\, dx, \\
	P_i(\sigma_i(s)) 
	& = \frac{N-2}{2}e^{(N-2)s}\|\nabla \omega_i\|_{L^2}^2 - Ne^{Ns}\int_{\R^N}F_i(\omega_i)\, dx.
\end{align*}
Since $\omega_i$ satisfies the Pohozaev identity $P_i(\omega_i)=0$, we have
\begin{equation*}
	\frac{d}{ds}J_i(\sigma_i(s)) 
	= P_i(\sigma_i(s)) 
	= \frac{N-2}{2}(1-e^{2s})e^{(N-2)s}\|\nabla \omega_i\|_{L^2}^2.
\end{equation*}
Thus (iv) holds and a function $s\mapsto J_i(\sigma_i(s))$ attains its maximum at $s=0$, which proves (iii).
\end{proof}

For $\vec{s}=(s_1,s_2,s_3)\in \R^3$, we define $\vec{\sigma}(\vec{s})=(\sigma_1(s_1), \sigma_2(s_2), \sigma_3(s_3))$.
Since $\vec{\sigma}(\vec{0})\in X$, we can choose $r>0$ and set $R=[-r,r]^3$ such that
\begin{equation*}
	\vec{\sigma}(\vec{s}) \in X_{\mu} \quad \text{ for all } \vec{s}\in R.
\end{equation*}
We define the constant $d_\alpha$ by 
\begin{equation*}
	d_\alpha = \max_{\vec{s}\in R}I_\alpha(\vec{\sigma}(\vec{s})).
\end{equation*}
Next, we define the minimax value $c_{\alpha}$ by
\begin{align*}
	& c_{\alpha}=\inf_{\vec{\gamma}\in \Gamma} \max_{\vec{s}\in R}I_\alpha(\vec{\gamma}(\vec{s})), \\
	& \Gamma = \left\{\vec{\gamma}\in C(R,\H)\,\left|\ \
	\begin{aligned}
	& \vec{\gamma}(\vec{s})=\vec{\sigma}(\vec{s}) \text{ for all } \vec{s}\in \partial R, \\
	& \vec{\gamma}(\vec{s})\in X_{2\mu} \text{ for all } \vec{s}\in R 
	\end{aligned} \right.\right\}.
\end{align*}
Clearly, $\vec{\sigma}\in \Gamma$ and $c_\alpha \leq d_\alpha$.
By the choice of $\mu$ and $r$, for any $\vec{\gamma}=(\gamma_1,\gamma_2,\gamma_3)\in \Gamma$, we have
\begin{equation*}
	\gamma_i(\vec{s}) \not=0 \quad \text{ for all } \vec{s}\in R \qquad (i=1,2,3).
\end{equation*}
Moreover, since $X_{2\mu}$ is bounded in $\H$, there exists a constant $D>0$ such that 
\begin{equation}\label{eq:3.4}
	\left|\int_{\R^N}\gamma_1(\vec{s})\gamma_2(\vec{s})\gamma_3(\vec{s})\,dx\right| \leq D \quad 
	\text{ for all $\vec{s}\in R$ and $\vec{\gamma}\in \Gamma$}.
\end{equation}
Since $\Gamma$ is not invariant under the gradient flow of $I_\alpha$, 
$c_{\alpha}$ is not necessarily a critical value. 
Nevertheless, it plays a key role in establishing the existence of a critical point in $X_{\mu}$.

The following property holds for $c_{\alpha}$.

\begin{lemma}\label{lem:3.4}
$\displaystyle \lim_{\alpha\to 0}c_{\alpha} = c_1 + c_2 + c_3$ 
and $\displaystyle \lim_{\alpha\to 0}d_{\alpha} = c_1 + c_2 + c_3$.
\end{lemma}

\begin{proof}
Let $\vec{s}_0\in R$ satisfy 
$\displaystyle d_\alpha=\max_{\vec{s}\in R}I_\alpha(\vec{\sigma}(\vec{s}))=I_\alpha(\vec{\sigma}(\vec{s}_0))$.
Then, by Lemma \ref{lem:3.3} (iii) and \eqref{eq:3.4}, we have
\begin{equation*}
	c_{\alpha}
	\leq d_\alpha = I_\alpha(\vec{\sigma}(\vec{s}_0))
	\leq \max_{\vec{s}\in R}\left(\sum_{i=1}^3J_i(\sigma_i(s_i))\right)  + |\alpha| D
	\leq c_1+c_2+c_3 + |\alpha| D.
\end{equation*}
Next, for $\vec{\gamma}=(\gamma_1,\gamma_2,\gamma_3)\in \Gamma$, we define two maps $\Phi, \Psi:R \to \R^3$ by
\begin{align*}
	\Phi(\vec{s}) & = (P_1(\gamma_1(\vec{s})),P_2(\gamma_2(\vec{s})), P_3(\gamma_3(\vec{s}))),\\
	\Psi(\vec{s}) & = (P_1(\sigma_1(s_1)),P_2(\sigma_2(s_2)), P_3(\sigma_3(s_3))).
\end{align*}
Since $\vec{\gamma}|_{\partial R}=\vec{\sigma}|_{\partial R}$, we have $\Phi|_{\partial R}=\Psi|_{\partial R}$.
By Lemma \ref{lem:3.3} (iv), we have
\begin{equation*}
	\Psi(\vec{s}) = \frac{N-2}{2}\left[\begin{aligned}
	& (1-e^{2s_1})e^{(N-2)s_1}\|\nabla \omega_1\|_{L^2}^2 \\
	& (1-e^{2s_2})e^{(N-2)s_2}\|\nabla \omega_2\|_{L^2}^2 \\
	& (1-e^{2s_3})e^{(N-2)s_3}\|\nabla \omega_3\|_{L^2}^2 
	\end{aligned}\right]^T.
\end{equation*}
A direct computation shows that $\deg(\Psi, (-r,r)^3, (0,0,0)) = -1$.
Hence
\begin{equation*}
	\deg(\Phi, (-r,r)^3, (0,0,0)) = \deg(\Psi,  (-r,r)^3, (0,0,0)) = -1.
\end{equation*}
Therefore, there exists $\vec{s}_0\in (-r,r)^3$ such that $\Phi(\vec{s}_0)=(0,0,0)$.
In particular, $P_i(\gamma_i(\vec{s}_0)) = 0$ for $i=1,2,3$.  
Recalling from \cite{JT} that $c_i$ is characterized by $\displaystyle c_i = \inf_{P_i(u)=0} J_i(u)$, we have
\begin{align*}
	\max_{\vec{s}\in R}I_\alpha(\vec{\gamma}(\vec{s}))
	& \geq I_\alpha(\vec{\gamma}(\vec{s}_0)) \\
	& \geq J_1(\gamma_1(\vec{s}_0)) + J_2(\gamma_2(\vec{s}_0)) + J_3(\gamma_3(\vec{s}_0)) - |\alpha| D \\
	& \geq c_1 + c_2 + c_3 - |\alpha| D.
\end{align*}
Since this holds for any $\vec{\gamma} \in \Gamma$, we obtain
\begin{equation*}
	c_{\alpha} \geq c_1 + c_2 + c_3 - |\alpha| D.
\end{equation*}
Consequently, it follows that $\displaystyle \lim_{\alpha \to 0}c_{\alpha} = \lim_{\alpha \to 0}d_{\alpha} = c_1+c_2+c_3$.
\end{proof}

\begin{lemma}\label{lem:3.5}
There exists $\alpha_2, \delta_2>0$ such that, for any $\alpha$ with $|\alpha|\leq \alpha_2$, 
\begin{equation}\label{eq:3.5}
	\max_{\vec{s} \in \partial R}I_\alpha(\vec{\sigma}(\vec{s})) \leq c_1+c_2+c_3 - \delta_2. 
\end{equation}
\end{lemma}

\begin{proof}
By Lemma \ref{lem:3.3} (iii), there exists $\delta_2>0$ such that
\begin{equation*}
	\max_{\vec{s} \in \partial R } \sum_{i=1}^3 J_i(\sigma_i(s_i)) \leq c_1+c_2+c_3 - 2\delta_2.
\end{equation*}
By \eqref{eq:3.4}, we have 
\begin{equation*}
	I_\alpha(\vec{\sigma}(\vec{s})) \leq \sum_{i=1}^3 J_i(\sigma_i(s_i)) + |\alpha|D.
\end{equation*}
Hence there exists $\alpha_2>0$ such that, for any $\alpha$ with $|\alpha|\leq \alpha_2$, 
the inequality \eqref{eq:3.5} holds.
\end{proof}

\subsection{The existence and asymptotic behavior of critical points}

Let $\alpha_1 > 0$, $\delta_1 > 0$ and $\rho > 0$ be the constants given in Proposition \ref{prop:3.2}, 
and let  $\alpha_2 > 0$, $\delta_2 > 0$ be the constants in Lemma \ref{lem:3.5}.
Set $\delta_0 = \min\{\delta_1, \delta_2, \frac{\rho\mu}{2}\}$.  
Then, by Lemma~\ref{lem:3.4}, there exists $0<\alpha_0<\min\{\alpha_1,\alpha_2\}$ such that
\begin{equation*}
	c_1 + c_2 + c_3 - \delta_0 < c_{\alpha}\leq d_\alpha < c_1 + c_2 + c_3 + \delta_0 \quad \text{ for all } |\alpha| \leq \alpha_0.
\end{equation*}

In this situation, we have the following proposition.

\begin{proposition}\label{prop:3.6}
For every $\alpha$ with $|\alpha| \leq \alpha_0$, it holds that
\begin{equation*}
	\inf_{\vec{u}\in X_{\mu}\cap[I_\alpha\leq d_\alpha]}\|I_\alpha'(\vec{u})\|_{\H^*}=0.
\end{equation*}
\end{proposition}

\begin{proof}
We argue by contradiction.  
Suppose that Proposition \ref{prop:3.6} does not hold. 
Then there exists $\alpha$ with $|\alpha|\leq \alpha_0$ such that
\begin{equation}\label{eq:3.6}
	\inf_{\vec{u}\in X_{\mu}\cap[I_\alpha\leq d_\alpha]}\|I_\alpha'(\vec{u})\|_{\H^*}=:\rho_\alpha>0.
\end{equation}
We choose a pseudo-gradient vector field $W$ for $I_\alpha$.
It is a locally Lipschitz continuous mapping
$ W:\left\{\vec{u}\in \H \mid I_\alpha'(\vec{u})\neq 0\right\} \longrightarrow \H\setminus\{0\}$
and satisfies
\begin{equation*}
    \|W(\vec{u})\|_{\H}\leq 2\|I_\alpha'(\vec{u})\|_{\H^*},
    \qquad
    \langle I_\alpha'(\vec{u}), \, W(\vec{u}) \rangle_{\H^*, \H}
    \geq \|I_\alpha'(\vec{u})\|_{\H^*}^{2}.
\end{equation*}
Recall that $\vec{\sigma}(\vec{s})\in X_\mu$ for all $\vec{s}\in R$.
For $\vec{s}\in R$, we consider the following differential equation in $\H$ with initial value 
$\vec{\sigma}(\vec{s})$:
\begin{equation*}
	\left\{ \begin{aligned}
	& \frac{d\vec{\eta}}{dt}(t;\vec{\sigma}(\vec{s})) 
	= -\frac{W(\vec{\eta}(t;\vec{\sigma}(\vec{s})))}{\|W(\vec{\eta}(t;\vec{\sigma}(\vec{s})))\|_{\H}}, \\
	& \vec{\eta}(0;\vec{\sigma}(\vec{s})) = \vec{\sigma}(\vec{s}).
	\end{aligned}\right.
\end{equation*}
If the initial value satisfies $\vec{\sigma}(\vec{s})\in [I_\alpha>c_1+c_2+c_3-\delta_0]$,
then as long as the solution $\vec{\eta}(t;\vec{\sigma}(\vec{s}))$ remains in $X_{2\mu}\cap [I_\alpha>c_1+c_2+c_3-\delta_0]$, we have
\begin{equation}\label{eq:3.7}
	\|\vec{\eta}(t;\vec{\sigma}(\vec{s})) - \vec{\sigma}(\vec{s}) \|_{\H}
	\leq \int_0^t\left\|\frac{d}{d\tau}\vec{\eta}(\tau;\vec{\sigma}(\vec{s}))\right\|_{\H}\, d\tau
	= t
\end{equation}
and
\begin{align}
	I_\alpha(\vec{\eta}(t;\vec{\sigma}(\vec{s}))) - I_\alpha(\vec{\sigma}(\vec{s}))
	& = \int_0^t \frac{d}{d\tau}I_\alpha(\vec{\eta}(\tau;\vec{\sigma}(\vec{s})))\, d\tau \notag \\
	& = \int_0^t I_\alpha'(\vec{\eta}(\tau;\vec{\sigma}(\vec{s})))\frac{d\vec{\eta}}{d\tau}(\tau;\vec{\sigma}(\vec{s})) \, d\tau \notag \\
	& \leq -\int_0^t \|I_\alpha'(\vec{\eta}(\tau;\vec{\sigma}(\vec{s})))\|_{\H^*} \, d\tau. \label{eq:3.8}
\end{align}
It follows from \eqref{eq:3.6}, \eqref{eq:3.8}, and Proposition \ref{prop:3.2} that
$\vec{\eta}$ reaches the boundary $\partial (X_{2\mu}\cap [I_\alpha>c_1+c_2+c_3-\delta_0])$ in finite time.  
Let $t(\vec{s})$ denote the first such time.
If $\vec{\sigma}(\vec{s}) \notin [I_\alpha > c_1 + c_2 + c_3 - \delta_0]$, we set $t(\vec{s}) = 0$.  
We now claim the following.

\medskip

\noindent
\textbf{Claim.}  
If the initial value satisfies $\vec{\sigma}(\vec{s}) \in [I_\alpha > c_1 + c_2 + c_3 - \delta_0]$,  
then $\vec{\eta}$ reaches the level set $[I_\alpha = c_1 + c_2 + c_3 - \delta_0]$ before reaching $\partial X_{2\mu}$.

\medskip

Indeed, suppose instead that $\vec{\eta}$ stops upon reaching $\partial X_{2\mu}$.  
Then we can find times $t_1 < t_2$ such that 
$\vec{\eta}(t_1;\vec{\sigma}(\vec{s}))\in \partial X_{\mu}$ and $\vec{\eta}(t_2;\vec{\sigma}(\vec{s}))\in \partial X_{2\mu}$.
Arguing as in \eqref{eq:3.7}, we have
\begin{equation*}
	\mu\leq \|\vec{\eta}(t_2;\vec{\sigma}(\vec{s})) - \vec{\eta}(t_1;\vec{\sigma}(\vec{s})) \|_{\H}
	\leq t_2-t_1.
\end{equation*}
Moreover, it follows from \eqref{eq:3.8} that
\begin{equation*}
	I_\alpha(\vec{\eta}(t_2;\vec{\sigma}(\vec{s}))) - I_\alpha(\vec{\sigma}(\vec{s}))
	\leq -\int_{t_1}^{t_2} \|I_\alpha'(\vec{\eta}(\tau;\vec{\sigma}(\vec{s})))\|_{\H^*} \, d\tau 
	\leq -\rho(t_2-t_1)
	\leq -\rho\mu 
	\leq -2\delta_0,
\end{equation*}
where $\rho>0$ is the constant given in Proposition \ref{prop:3.2}.
Hence,
\begin{equation*}
	I_\alpha(\vec{\eta}(t_2;\vec{\sigma}(\vec{s}))) \le c_1 + c_2 + c_3 - \delta_0,
\end{equation*}
which means that $\vec{\eta}$ must have reached the level set $[I_\alpha = c_1 + c_2 + c_3 - \delta_0]$ earlier.
This is a contradiction, proving the Claim.

\medskip

From the above Claim, we obtain
\begin{equation}\label{eq:3.9}
	\max_{\vec{s}\in R}I_\alpha(\vec{\eta}(t(\vec{s});\vec{\sigma}(\vec{s})))\leq c_1+c_2+c_3-\delta_0.
\end{equation}
Let $\vec{\gamma}(\vec{s})=\vec{\eta}(t(\vec{s});\vec{\sigma}(\vec{s}))$.
Applying the implicit function theorem to 
\begin{equation*}
	f(t,\vec{s}) = I_\alpha(\vec{\eta}(t,\vec{\sigma}(\vec{s}))) = c_1 + c_2 + c_3 - \delta_0,
\end{equation*}
we conclude that $t(\vec{s})$ is continuous.
Hence we have $\vec{\gamma}\in C(R,\H)$.
For $\vec{s}\in\partial R$, we have $\vec{\sigma}(\vec{s})\in [I_\alpha \le c_1 + c_2 + c_3 - \delta_0]$ by \eqref{eq:3.5}.
Hence $\vec{\gamma}(\vec{s}) = \vec{\sigma}(\vec{s})$ for all $\vec{s}\in\partial R$.
Moreover, the above Claim implies $\vec{\gamma}(\vec{s}) = \vec{\eta}(t(\vec{s});\vec{\sigma}(\vec{s})) \in X_{2\mu}$ for all $\vec{s}\in R$.
Consequently, we have $\vec{\gamma}\in \Gamma$.
Combining this with \eqref{eq:3.9}, we reach a contradiction:
\begin{equation*}
	c_{\alpha}\leq \max_{\vec{s}\in R}I_\alpha(\vec{\gamma}(\vec{s}))\leq c_1+c_2+c_3-\delta_0 < c_{\alpha}.
\end{equation*}
This proves Proposition \ref{prop:3.6}.
\end{proof}

\begin{proof}[Proof of Theorem \ref{thm:1.1}.]
Fix any $\alpha$ with $|\alpha| \leq \alpha_0$. 
By Proposition \ref{prop:3.6}, there exists a sequence $\{\vec{u}_n\}$ satisfying
\begin{equation*}
	\vec{u}_n \in X_{\mu}\cap[I_\alpha\leq d_\alpha] \quad \text{ and } \quad 
	\|I_{\alpha}'(\vec{u}_n)\|_{\H^*}\to 0.
\end{equation*}
By Proposition \ref{prop:2.2}, $\{\vec{u}_n\}$ has a subsequence converging strongly in $\H$ 
to some $\vec{u}_\alpha \in X_\mu \cap [I_\alpha\leq d_\alpha]$. 
In particular, $\vec{u}_\alpha$ is a vector solution of \eqref{eq:1.1}.
For the family $\{\vec{u}_\alpha\}$, Lemma \ref{lem:3.1} implies that $\displaystyle \lim_{\alpha \to 0}{\rm dist}(\vec{u}_\alpha,X)=0$.
\end{proof}

\section{Proof of Theorems \ref{thm:1.2} and \ref{thm:1.3}}

In this section, we prove Theorems \ref{thm:1.2} and \ref{thm:1.3}.

\subsection{Proof of Theorem \ref{thm:1.2}}
The proof of Theorem \ref{thm:1.2} is almost identical to that of Theorem \ref{thm:1.1}, so we only outline the main steps here.  
We choose a constant $\tilde{\mu}$ satisfying
\begin{equation*}
	0<\tilde{\mu}<\frac{\,1\,}{3}\inf\left\{\|\omega_i\|_{H^1}\,|\, \omega_i\in S_i \ (i=1,2) \right\} \ \text{ and } \
	J_3(u_3) \geq 0 \text{ for all } \|u_3\|_{H^1}\leq 2\tilde{\mu}.
\end{equation*}
We recall that $Y=S_1 \times S_2 \times \{0\}$ and define the $\tilde{\mu}$-neighborhood of $Y$ by
\begin{equation*}
	Y_{\tilde{\mu}}=\{\vec{u} \in \H\,|\, {\rm dist}(\vec{u}, Y)\leq \tilde{\mu}\}.
\end{equation*}
By this choice of $\tilde{\mu}$, the first and second components of any element in $Y_{2\tilde{\mu}}$ are nonzero.  

\begin{lemma}\label{lem:4.1}
Let $\{\alpha_n\}$ and $\{\delta_n\}$ be sequences with $\alpha_n \to 0$ and $\delta_n \to 0^+$.
Suppose that a sequence $\{\vec{u}_n\}$ satisfies
\begin{equation}\label{eq:4.1}
	\vec{u}_n \in Y_{2\tilde{\mu}}\cap [I_{\alpha_n}\leq c_1+c_2+\delta_n] \quad \text{ and } \quad 
	\|I_{\alpha_n}'(\vec{u}_n)\|_{\H^*}\to 0 \text{ as } n\to \infty. 
\end{equation}
Then 
\begin{equation}\label{eq:4.2}
	\lim_{n\to \infty}{\rm dist}(\vec{u}_n, Y)=0.
\end{equation}
\end{lemma}

\begin{proof}
Suppose that \eqref{eq:4.2} does not hold. 
Then, up to a subsequence, we have
\begin{equation}\label{eq:4.3}
	\lim_{n\to \infty}{\rm dist}(\vec{u}_n, Y)>0.
\end{equation}
Since $Y_{2\tilde{\mu}}$ is bounded, Proposition \ref{prop:2.3} implies that, up to a subsequence,
$\vec{u}_n \to \vec{u}_0$ strongly in $\H$ for some $\vec{u}_0=(u_{10},u_{20},u_{30})$. 
Then $\vec{u}_0\in Y_{2\tilde{\mu}}$ and $I_0'(\vec{u}_0)=0$.
These mean that $u_{10}$ and $u_{20}$ are nontrivial critical points of $J_1$ and $J_2$ respectively, 
and $u_{30}$ is a critical point of $J_3$.
Therefore
\begin{equation*}
	\lim_{n\to \infty}I_{\alpha_n}(\vec{u}_n) = \sum_{i=1}^3J_i(u_{i0})\geq c_1+c_2.
\end{equation*}
Since ${\displaystyle \varlimsup_{n\to \infty}}I_{\alpha_n}(\vec{u}_n)\leq c_1+c_2$, we deduce that 
$J_i(u_{i0})=c_i$ for $i=1,2$ and $J_3(u_{30})=0$.
Thus, $\vec{u}_0\in Y$, which contradicts \eqref{eq:4.3}.
Hence, \eqref{eq:4.2} holds.
\end{proof}

\begin{proposition}\label{prop:4.2}
There exist constants $\tilde{\alpha}_1>0$, $\tilde{\delta}_1>0$ and $\tilde{\rho}>0$ such that, for all $\alpha$ with $|\alpha|\leq \tilde{\alpha}_1$, we have
\begin{equation*}
	\|I_\alpha'(\vec{u})\|_{\H^*}\geq \tilde{\rho} \quad \text{ for all } \vec{u}\in (Y_{2\tilde{\mu}}\setminus Y_{\tilde{\mu}})\cap[I_\alpha\leq c_1+c_2+\tilde{\delta}_1]
\end{equation*}
\end{proposition}

\begin{proof}
We argue by contradiction.
Suppose the statement is false.
Then there exist sequences $\{\alpha_n\}$ and $\{\delta_n\}$ with $\alpha_n\to 0$ and $\delta_n\to 0^+$, and a sequence $\{\vec{u}_n\}$ such that 
\begin{equation*}
	\vec{u}_n \in (Y_{2\tilde{\mu}}\setminus Y_{\tilde{\mu}})\cap [I_{\alpha_n} \leq c_1+c_2+\delta_n] \quad \text{ and } \quad 
	\|I_{\alpha_n}'(\vec{u}_n)\|_{\H^*}\to 0.
\end{equation*}
By Lemma \ref{lem:4.1}, we obtain $\displaystyle \lim_{n\to \infty}{\rm dist}(\vec{u}_n, Y)=0$,
which contradicts $\vec{u}_n\notin Y_{\tilde{\mu}}$.
This contradiction completes the proof.
\end{proof}

For $\vec{r}=(r_1,r_2)\in \R^2$, we define $\vec{\tau}(\vec{r})=(\sigma_1(r_1), \sigma_2(r_2), 0)$.
Since $\vec{\tau}(\vec{0})\in Y$, we can choose $l>0$ and set $L=[-l,l]^2$ such that
\begin{equation*}
	\vec{\tau}(\vec{r})  \in Y_{\tilde{\mu}} \quad \text{ for all } \vec{r}\in L.
\end{equation*}
We define the minimax value $b_{\alpha}$ by
\begin{align*}
	& b_{\alpha}=\inf_{\vec{\gamma}\in \widetilde{\Gamma}}\max_{\vec{r}\in L}I_\alpha(\vec{\gamma}(\vec{r})), \\
	& \widetilde{\Gamma} = \left\{\vec{\gamma}\in C(L,\H)\,\left|\ \
	\begin{aligned}
	& \vec{\gamma}(\vec{r})=\vec{\tau}(\vec{r}) \text{ for all } \vec{r}\in \partial L, \\
	& \vec{\gamma}(\vec{r})\in Y_{2\tilde{\mu}} \text{ for all } \vec{r}\in L 
	\end{aligned} \right.\right\}.
\end{align*}
Clearly, $\vec{\tau}=(\sigma_1, \sigma_2, 0)\in \widetilde{\Gamma}$ and
\begin{equation}\label{eq:4.4}
	I_{\alpha}(\vec{\tau}(\vec{r})) = J_1(\sigma_1(r_1)) + J_2(\sigma_2(r_2)) \quad \text{ for all } \vec{r}\in L.
\end{equation}
By the choice of $\tilde{\mu}$ and $l$, for any $\vec{\gamma}=(\gamma_1,\gamma_2,\gamma_3)\in \widetilde{\Gamma}$, we have
\begin{equation*}
	\gamma_i(\vec{r}) \not=0 \quad \text{ for all } \vec{r}\in L \qquad (i=1,2).
\end{equation*}
The following properties hold for $b_{\alpha}$.

\begin{lemma}\label{lem:4.3}
$b_{\alpha} \leq c_1 + c_2$ and $\displaystyle \lim_{\alpha\to 0}b_{\alpha} = c_1 + c_2$.
\end{lemma}

\begin{proof}
From \eqref{eq:4.4}, it follows that 
\begin{equation*}
	b_{\alpha} \leq \max_{\vec{r}\in L}I_{\alpha}(\vec{\tau}(\vec{r})) \leq c_1+c_2.
\end{equation*}
For any $\vec{\gamma} = (\gamma_1, \gamma_2, \gamma_3) \in \widetilde{\Gamma}$, we define the two maps
\begin{align*}
	\tilde{\Phi}(\vec{r}) & = (P_1(\gamma_1(\vec{r})), P_2(\gamma_2(\vec{r}))): L \to \R^2,\\
	\tilde{\Psi}(\vec{r}) & = (P_1(\sigma_1(r_1)), P_2(\sigma_2(r_2))): L \to \R^2.
\end{align*}
Then, by the same argument as in the proof of Lemma \ref{lem:3.4},  
we have 
\begin{equation*}
	\deg(\tilde{\Phi}, (-l,l)^2, (0,0)) = \deg(\tilde{\Psi}, (-l,l)^2, (0,0)) = -1.
\end{equation*}
Hence there exists $\vec{r}_0 \in (-l,l)^2$ such that $P_i(\gamma_i(\vec{r}_0)) = 0$ for $i=1,2$.  
By $\displaystyle c_i = \inf_{P_i(u)=0} J_i(u)$ and \eqref{eq:3.4}, we have
\begin{align*}
	\max_{\vec{r}\in L}I_\alpha(\vec{\gamma}(\vec{r}))
	& \geq I_\alpha(\vec{\gamma}(\vec{r}_0)) \\
	& \geq J_1(\gamma_1(\vec{r}_0)) + J_2(\gamma_2(\vec{r}_0)) + J_3(\gamma_3(\vec{r}_0)) - |\alpha|\widetilde{D} \\
	& \geq c_1 + c_2 - |\alpha| \widetilde{D},
\end{align*}
for a constant $\widetilde{D}>0$.
Since this holds for any $\vec{\gamma} \in \widetilde{\Gamma}$, we obtain
\begin{equation*}
	b_{\alpha} \geq c_1 + c_2  - |\alpha| \widetilde{D}.
\end{equation*}
Consequently, it follows that $b_{\alpha} \to c_1+c_2$ $(\alpha \to 0)$.
\end{proof}

\begin{lemma}\label{lem:4.4}
There exists $\tilde{\delta}_2>0$ such that, for any $\alpha\in \R$ such that
\begin{equation}\label{eq:4.5}
	\max_{\vec{r} \in \partial L} I_\alpha(\vec{\tau}(\vec{r})) \leq c_1+c_2 - \tilde{\delta}_2. 
\end{equation}
\end{lemma}

\begin{proof}
By Lemma \ref{lem:3.3} (iii), there exists $\tilde{\delta}_2 > 0$ such that
\begin{equation*}
	\max_{\vec{r} \in \partial L } \sum_{i=1}^2 J_i(\sigma_i(s_i)) \leq c_1 + c_2 - 2\tilde{\delta}_2.
\end{equation*}
Hence \eqref{eq:4.5} follows from \eqref{eq:4.4}.
\end{proof}

Let $\tilde{\alpha}_1>0$ and $\tilde{\rho}>0$ be the constants given in Proposition~\ref{prop:4.2}, 
and let $\tilde{\delta}_2>0$ be the constant in Lemma~\ref{lem:4.4}.  
Set $\tilde{\delta}_0 = \min\{\tilde{\delta}_1,\tilde{\delta}_2, \tilde{\rho}\tilde{\mu}\}$. 
Then, by Lemma~\ref{lem:4.3}, there exists $\tilde{\alpha}_0 \in (0, \tilde{\alpha}_1]$ such that
\begin{equation*}
	c_1 + c_2 - \tilde{\delta}_0 < b_{\alpha} \leq c_1 + c_2 \quad \text{ for all } |\alpha| \leq \tilde{\alpha}_0.
\end{equation*}
We now have the following result.

\begin{proposition}\label{prop:4.5}
For every $\alpha$ with $|\alpha| \leq \tilde{\alpha}_0$, it holds that
\begin{equation*}
	\inf_{\vec{u}\in Y_{\tilde{\mu}}\cap[I_\alpha\leq c_1+c_2]}\|I_\alpha'(\vec{u})\|_{\H^*}=0.
\end{equation*}
\end{proposition}

\begin{proof}
Suppose that Proposition \ref{prop:4.5} is false. 
Then there exists $\alpha$ with $|\alpha| \leq \tilde{\alpha}_0$ such that
\begin{equation*}
	\inf_{\vec{u}\in Y_{\tilde{\mu}}\cap[I_\alpha\leq c_1+c_2]}\|I_\alpha'(\vec{u})\|_{\H^*}:=\rho_\alpha>0.
\end{equation*}
As in Proposition \ref{prop:3.6}, we choose a pseudo-gradient vector field $W$ for $I_\alpha$.
For $\vec{r}\in L$, we consider the following differential equation in $\H$ with initial value 
$\vec{\tau}(\vec{r})$:
\begin{equation*}
	\left\{ \begin{aligned}
	& \frac{d\vec{\eta}}{dt}(t;\vec{\tau}(\vec{r})) 
	= -\frac{W(\vec{\eta}(t;\vec{\tau}(\vec{r})))}{\|W(\vec{\eta}(t;\vec{\tau}(\vec{r})))\|_{\H}}, \\
	& \vec{\eta}(0;\vec{\tau}(\vec{r})) = \vec{\tau}(\vec{r}).
	\end{aligned}\right.
\end{equation*}
By the same arguments as in Proposition \ref{prop:3.6}, 
if the initial value satisfies $\vec{\tau}(\vec{r})\in [I_\alpha>c_1+c_2-\tilde{\delta}_0]$,    
then $\vec{\eta}$ reaches the level set $[I_\alpha = c_1 + c_2 - \tilde{\delta}_0]$ before reaching $\partial Y_{2\tilde{\mu}}$.
Let $t(\vec{r})$ denote this first hitting time.
If $\vec{\tau}(\vec{r}) \notin [I_\alpha > c_1 + c_2 - \tilde{\delta}_0]$, we set $\tilde{t}(\vec{r}) = 0$.  
Then, we obtain
\begin{equation*}
	\max_{\vec{r}\in L}I_\alpha(\vec{\eta}(\tilde{t}(\vec{r});\vec{\tau}(\vec{r})))\leq c_1+c_2-\tilde{\delta}_0.
\end{equation*}
Set $\vec{\gamma}(\vec{r})=\vec{\eta}(\tilde{t}(\vec{r});\vec{\tau}(\vec{r}))$.
By arguments analogous to those used in Proposition \ref{prop:3.6}, we can verify that $\vec{\gamma}\in \widetilde{\Gamma}$.
Therefore, 
\begin{equation*}
	b_{\alpha}\leq \max_{\vec{r}\in L}I_\alpha(\vec{\gamma}(\vec{r}))\leq c_1+c_2-\tilde{\delta}_0 < b_{\alpha},
\end{equation*}
which is a contradiction.
Hence Proposition \ref{prop:4.5} is true.
\end{proof}

\begin{lemma}\label{lem:4.6}
Let $\vec{u}=(u_1,u_2,u_3)$ be a solution of \eqref{eq:1.1}.
If $u_1 \not= 0$ and $u_2 \not= 0$, then $u_3\not=0$.
\end{lemma}

\begin{proof}
We argue by contradiction.  
Suppose that $u_3 = 0$. 
Then, for any $\varphi \in H^1(\R^N)$,
\begin{equation*}
	I_\alpha'(u_1,u_2,0)(0,0,\varphi) = - \alpha \int_{\R^N}u_1u_2\varphi\, dx=0.
\end{equation*}
This implies $u_1 u_2 = 0$ in $\R^N$. 
For $i=1,2$, $u_i$ is a nontrivial solution of 
\begin{equation*}
	-\Delta u_i + V_i(x) u_i = 0, \qquad V_i(x)=-\frac{f_i(u_i(x))}{u_i(x)}.
\end{equation*}
Then, we have $V_i\in L^{\frac N2}_{\rm loc}(\R^N)$.  
By the strong unique continuation property for Schr\"odinger operators with potentials in $L^{\frac N2}_{\rm loc}$ (see \cite{JK}),
any solution that vanishes in a nonempty open set must vanish identically.
Since $u_1u_2=0$ in $\R^N$, at least one of $u_1$ or $u_2$ must vanish on a nonempty open set, and hence that function must be identically zero.
This contradicts the assumption that $u_1\neq0$ and $u_2\neq0$.
Therefore $u_3\neq0$, and the proof is completed.
\end{proof}

\begin{proof}[Proof of Theorem \ref{thm:1.2}.]
Fix $\alpha$ with $|\alpha|\leq \tilde{\alpha}_0$.  
By Proposition \ref{prop:4.5}, there exists a sequence 
$\{\vec{u}_n\}$ such that 
\begin{equation*}
	\vec{u}_n \in Y_{\tilde{\mu}} \cap [I_\alpha \leq c_1+c_2] \quad \text{ and } \quad 
	\|I_{\alpha}'(\vec{u}_n)\|_{\H^*} \to 0.
\end{equation*}  
By Proposition \ref{prop:2.2}, $\{\vec{u}_n\}$ has a subsequence converging strongly in $\H$ to some $\vec{u}_\alpha \in Y_{\tilde{\mu}}$.  
In particular, by Lemma \ref{lem:4.6}, $\vec{u}_\alpha$ is a vector solution of \eqref{eq:1.1}.  
Moreover, for the family $\{\vec{u}_\alpha\}$, Lemma \ref{lem:4.1} implies that 
$\displaystyle \lim_{\alpha \to 0}{\rm dist}(\vec{u}_\alpha,Y)=0$.
\end{proof}

\subsection{Proof of Theorem \ref{thm:1.3}}

For $\lambda>0$, we define
\begin{equation*}
	\|u\|_{\lambda}^2 = \|\nabla u\|_{L^2}^2 + \lambda \|u\|_{L^2}^2.
\end{equation*}
We first establish the following lemma.

\begin{lemma}\label{lem:4.7}
Let $M>0$. 
There exist constants $\alpha_M>0$ and $\rho_0>0$ such that, for any $\alpha$ with $|\alpha|\leq \alpha_M$, 
every solution $\vec{u}=(u_1,u_2,u_3)$ of \eqref{eq:1.1} with $\|\vec{u}\|_{\H}\leq M$ satisfies
\begin{align}
	& (u_1,u_2) \not = (0,0) \ \ \Longrightarrow \ \ \|u_1\|_{\lambda_1}^2 + \|u_2\|_{\lambda_2}^2 \geq \rho_0, \label{eq:4.6}\\
	& (u_1,u_3) \not = (0,0) \ \ \Longrightarrow \ \ \|u_1\|_{\lambda_1}^2 + \|u_3\|_{\lambda_3}^2 \geq \rho_0, \label{eq:4.7}\\
	& (u_2,u_3) \not = (0,0) \ \ \Longrightarrow \ \ \|u_2\|_{\lambda_2}^2 + \|u_3\|_{\lambda_3}^2 \geq \rho_0. \label{eq:4.8}
\end{align}
\end{lemma}

\begin{proof}
We prove \eqref{eq:4.8}. 
\eqref{eq:4.6} and \eqref{eq:4.7} also can be shown in a similar way.
Since $h_i^+(\xi)$ satisfies \eqref{eq:2.2} and (f1), there exists a constant $C>0$ such that
\begin{equation*}
	|h_i^+(\xi)\xi|\leq C|\xi|^{2^*} \quad \text{ for all } \xi\in \R.
\end{equation*}
Hence 
\begin{equation*}
	g_i(\xi)\xi = h_i^+(\xi)\xi- h_i^-(\xi)\xi \leq h_i^+(\xi)\xi \leq C|\xi|^{2^*} \quad \text{ for all } \xi\in \R.
\end{equation*}
From $I_{\alpha}'(\vec{u})(0,u_2,u_3)=0$, we have
\begin{align*}
	\|u_2\|_{\lambda_2}^2 + \|u_3\|_{\lambda_3}^2
	& = \int_{\R^N}g_2(u_2)u_2 + g_3(u_3)u_3 \,dx + 2\alpha \int_{\R^N}u_1u_2u_3\, dx \\
	& \leq C\left(\|u_2\|_{L^{2^*}}^{2^*} + \|u_3\|_{L^{2^*}}^{2^*}\right) + 2|\alpha| \|u_1\|_{L^3}\|u_2\|_{L^3}\|u_3\|_{L^3} \\
	& \leq C'(\|u_2\|_{\lambda_2}^2 + \|u_3\|_{\lambda_3}^2)^{\frac{2^*}{2}}
	+ |\alpha| C'M\left(\|u_2\|_{\lambda_2}^2+\|u_3\|_{\lambda_3}^2\right),
\end{align*}
where $C' > 0$ is a constant which does not depend on $M>0$.  
Therefore, we obtain
\begin{align*}
	1 - |\alpha| C'M \leq C'(\|u_2\|_{\lambda_2}^2 + \|u_3\|_{\lambda_3}^2)^{\frac{2^*}{2}-1}.
\end{align*}
We choose $\alpha_M>0$ and $\rho_0>0$ such that $1 - \alpha_M C'M = \frac{1}{2}$, $\rho_0=\left(2C'\right)^{-\frac{2}{2^*-2}}$.
Then \eqref{eq:4.8} follows.
\end{proof}

\begin{proof}[Proof of Theorem \ref{thm:1.3}]
Let $M=\sup_{\omega_1\in S_1}\|\omega_1\|_{H^1}+1$.
By Lemma \ref{lem:4.7}, for any $\alpha$ with $|\alpha| \leq \alpha_M$, \eqref{eq:1.1} has no vector solutions $\vec{u}=(u_1,u_2,u_3)$ 
satisfying $\|\vec{u}\|_{\H}\leq M$ and $\|u_2\|_{\lambda_2}+\|u_3\|_{\lambda_3}<\rho_0$.
Consequently, there exists no family of solutions to \eqref{eq:1.1} satisfying \eqref{eq:1.4}.
\end{proof}

\begin{remark}\label{rmk:4.8}
The nonexistence of the family of solutions $\{\vec{u}_{\beta}\}$ to \eqref{eq:1.5} satisfying \eqref{eq:1.6} 
also follows from the next claim.

\medskip

\noindent
\textbf{Claim.}  
Let $M>0$. 
There exist constants $\beta_M>0$ and $\rho>0$ such that, for any $\vec{\beta}=(\beta_{12},\beta_{13},\beta_{23})$ with 
$|\vec{\beta}|\leq \beta_M$, 
every solution $\vec{u}=(u_1,u_2,u_3)$ of \eqref{eq:1.5} with $\|\vec{u}\|_{\H}\leq M$ satisfies
\begin{equation*}
	u_i \not=0 \ \ \Longrightarrow \ \ \|u_i\|_{\lambda_i}^2 \geq \rho.
\end{equation*}

\medskip

\noindent
This claim can be proved in the same way as Lemma \ref{lem:4.7}.
\end{remark}

\section*{Acknowledgements}

This work was supported by JSPS KAKENHI Grant Numbers JP20K03691 and JP24KJ2070.

\bigskip

\noindent
{\bf Author Contributions} All authors contributed equally to the preparation and review of the manuscript.

\bigskip

\noindent
{\bf Data Availability Statement} This manuscript has no associated data.

\section*{Declarations}

\noindent
{\bf Conflict of Interest} The authors declare that they have no conflict of interest.


\begin{bibdiv}
\begin{biblist}



\bib{AM}{book}{
   author={Ambrosetti, Antonio},
   author={Malchiodi, Andrea},
   title={Nonlinear analysis and semilinear elliptic problems},
   series={Cambridge Studies in Advanced Mathematics},
   volume={104},
   publisher={Cambridge University Press, Cambridge},
   date={2007},
}


\bib{A}{article}{
   author={Ardila, Alex H.},
   title={Orbital stability of standing waves for a system of nonlinear
   Schr\"odinger equations with three wave interaction},
   journal={Nonlinear Anal.},
   volume={167},
   date={2018},
   pages={1--20},
}


\bib{BJ1}{article}{
   author={Byeon, Jaeyoung},
   author={Jeanjean, Louis},
   title={Standing waves for nonlinear Schr\"odinger equations with a
   general nonlinearity},
   journal={Arch. Ration. Mech. Anal.},
   volume={185},
   date={2007},
   number={2},
   pages={185--200},
}

\bib{BL1}{article}{
   author={Berestycki, H.},
   author={Lions, P.-L.},
   title={Nonlinear scalar field equations. I. Existence of a ground state},
   journal={Arch. Rational Mech. Anal.},
   volume={82},
   date={1983},
   number={4},
   pages={313--345},
}


\bib{CC1}{article}{
   author={Colin, M.},
   author={Colin, T.},
   title={A numerical model for the Raman amplification for laser-plasma
   interaction},
   journal={J. Comput. Appl. Math.},
   volume={193},
   date={2006},
   number={2},
   pages={535--562},
}

\bib{CC2}{article}{
   author={Colin, M.},
   author={Colin, T.},
   title={On a quasilinear Zakharov system describing laser-plasma
   interactions},
   journal={Differential Integral Equations},
   volume={17},
   date={2004},
   number={3-4},
   pages={297--330},
}

\bib{CCO1}{article}{
   author={Colin, M.},
   author={Colin, Th.},
   author={Ohta, M.},
   title={Stability of solitary waves for a system of nonlinear
   Schr\"odinger equations with three wave interaction},
   language={English, with English and French summaries},
   journal={Ann. Inst. H. Poincar\'e{} C Anal. Non Lin\'eaire},
   volume={26},
   date={2009},
   number={6},
   pages={2211--2226},
}

\bib{CCO2}{article}{
   author={Colin, Mathieu},
   author={Colin, Thierry},
   author={Ohta, Masahito},
   title={Instability of standing waves for a system of nonlinear
   Schr\"odinger equations with three-wave interaction},
   journal={Funkcial. Ekvac.},
   volume={52},
   date={2009},
   number={3},
   pages={371--380},
}

\bib{CO1}{article}{
   author={Colin, Mathieu},
   author={Ohta, Masahito},
   title={Bifurcation from semitrivial standing waves and ground states for
   a system of nonlinear Schr\"odinger equations},
   journal={SIAM J. Math. Anal.},
   volume={44},
   date={2012},
   number={1},
   pages={206--223},
}



\bib{CZ1}{article}{
   author={Chen, Z.},
   author={Zou, W.},
   title={On coupled systems of Schr\"odinger equations},
   journal={Adv. Differential Equations},
   volume={16},
   date={2011},
   number={7-8},
   pages={775--800},
}
\bib{CZ2}{article}{
   author={Chen, Zhijie},
   author={Zou, Wenming},
   title={On linearly coupled Schr\"odinger systems},
   journal={Proc. Amer. Math. Soc.},
   volume={142},
   date={2014},
   number={1},
   pages={323--333},
}





\bib{HIT}{article}{
   author={Hirata, Jun},
   author={Ikoma, Norihisa},
   author={Tanaka, Kazunaga},
   title={Nonlinear scalar field equations in $\mathbb R^N$: mountain pass and
   symmetric mountain pass approaches},
   journal={Topol. Methods Nonlinear Anal.},
   volume={35},
   date={2010},
   number={2},
   pages={253--276},
}


\bib{JK}{article}{
   author={Jerison, David},
   author={Kenig, Carlos E.},
   title={Unique continuation and absence of positive eigenvalues for
   Schr\"odinger operators},
   journal={Ann. of Math. (2)},
   volume={121},
   date={1985},
   number={3},
   pages={463--494},
}

\bib{JT}{article}{
   author={Jeanjean, Louis},
   author={Tanaka, Kazunaga},
   title={A remark on least energy solutions in ${\bf R}^N$},
   journal={Proc. Amer. Math. Soc.},
   volume={131},
   date={2003},
   number={8},
   pages={2399--2408},
}

\bib{KiO1}{article}{
   author={Kinoshita, Tomoharu},
   author={Osada, Yuki},
   title={Multiplicity of solutions for a nonlinear Schr\"odinger system
   with three wave interaction},
   journal={Partial Differ. Equ. Appl.},
   volume={6},
   date={2025},
   number={3},
   pages={Paper No. 22, 14},
}


\bib{KO1}{article}{
   author={Kurata, Kazuhiro},
   author={Osada, Yuki},
   title={Asymptotic expansion of the ground state energy for nonlinear Schr\"odinger system with three wave interaction},
   journal={Commun. Pure Appl. Anal.},
   volume={20},
   date={2021},
   number={12},
   pages={4239--4251},
}

\bib{KO2}{article}{
   author={Kurata, Kazuhiro},
   author={Osada, Yuki},
   title={Variational problems associated with a system of nonlinear Schr\"odinger equations with three wave interaction},
   journal={Discrete Contin. Dyn. Syst. Ser. B},
   volume={27},
   date={2022},
   number={3},
   pages={1511--1547},
}


\bib{O1}{article}{
   author={Osada, Yuki},
   title={Existence of a minimizer for a nonlinear Schr\"odinger system with
   three wave interaction under non-symmetric potentials},
   journal={Partial Differ. Equ. Appl.},
   volume={3},
   date={2022},
   number={2},
   pages={Paper No. 28, 18},
}

\bib{OS}{article}{
   author={Osada, Yuki},
   author={Sato, Yohei},
   title={A construction of peak solutions by a local mountain pass approach
   for a nonlinear Schr\"odinger system with three wave interaction},
   journal={Partial Differ. Equ. Appl.},
   volume={6},
   date={2025},
   number={1},
   pages={Paper No. 8, 26},
}


\bib{P1}{article}{
   author={Pomponio, A.},
   title={Ground states for a system of nonlinear Schr\"odinger equations
   with three wave interaction},
   journal={J. Math. Phys.},
   volume={51},
   date={2010},
   number={9},
   pages={093513, 20},
}




\end{biblist}
\end{bibdiv}
\end{document}